\documentclass[12pt]{article}
\usepackage{amsmath, amssymb, amsthm}

\newtheorem{thm}{Theorem}[section]

\makeatletter \@addtoreset{equation}{section} \makeatother

\begin{document}

\begin{center}
{\Large\bf The Real-Rootedness and Log-concavities of \\[6pt]
Coordinator Polynomials of Weyl Group Lattices}
\end{center}

\begin{center}
David G. L. Wang$^{1}$ and Tongyuan Zhao$^{2}$\\[6pt]

$^{1}$Beijing International Center for Mathematical Research\\
$^{2}$School of Mathematics, LMAM\\
$^{1,2}$Peking University, Beijing 100871, P.R. China

{\tt $^{1}$wgl@math.pku.edu.cn}$\quad $ {\tt
$^{2}$zhaotongyuan@pku.edu.cn}
\end{center}

\begin{abstract}
It is well-known that the coordinator polynomials of the classical root lattice of 
type~$A_n$ and those of type~$C_n$ are real-rooted. 
They can be obtained, either by the Aissen-Schoenberg-Whitney theorem, 
or from their recurrence relations.  
In this paper, we develop a trigonometric substitution approach 
which can be used to establish the real-rootedness 
of coordinator polynomials of type~$D_n$. 
We also find the coordinator polynomials of type $B_n$ 
are not real-rooted in general.
As a conclusion, 
we obtain that all coordinator polynomials of Weyl group lattices are log-concave.
\end{abstract}



\section{Introduction}

Let $f(x)=\sum_{i=1}^na_ix^i$ be a polynomial of degree $n$ with
nonnegative coefficients. We say that $f(x)$ is {\em real-rooted} if all
its zeros are real. 
Real-rooted polynomials have attracted much attention during the past decades.
One of the most significant reasons is that for any polynomial,
the real-rootedness implies the log-concavity of its coefficients,
which in turn implies the unimodality of the coefficients. 
Indeed, unimodal and log-concave sequences occur naturally 
in combinatorics, algebra, analysis, geometry,
computer science, probability and statistics. We refer the reader to the survey papers,
Brenti~\cite{Bre94} and Stanley~\cite{Sta89},
for various results on the unimodality and log-concavity.

There is a characterization of real-rooted
polynomials in the theory of total positivity; see
Karlin~\cite{Kar68}. 
A matrix $(a_{ij})_{i,j\ge0}$ is said to be {\em totally positive} if all its minors
have nonnegative determinants.
The sequence $\{a_k\}_{k=0}^n$ is called a
{\em P\'olya frequency sequence} if the lower triangular
matrix $(a_{i-j})_{i,j=0}^n$ is totally positive, where
$a_k$ is set to be zero if $k<0$. A basic link between P\'olya frequency sequences
and real-rooted polynomial was given by the Aissen-Schoenberg-Whitney
theorem~\cite{ASW52}, which stated that the polynomial $f(x)$ is real-rooted if and only
if the sequence $\{a_k\}_{k=0}^n$ is a P\'olya frequency
sequence. Another characterization from the probabilistic point of
view can be found in Pitman~\cite{Pit97}, see also Schoenberg~\cite{Sch55}.

Polynomials arising from combinatorics are often real-rooted. Basic
examples include the generating functions of binomial coefficients,
of Stirling numbers of the first kind and of the second kind, of
Eulerian numbers, and the matching polynomials; see, for example, Brenti~\cite{Bre95,Bre96},
Liu and Wang~\cite{LW07},
Stanley~\cite{Sta00} and Wang and Yeh~\cite{WY05}.

This paper is concerned with the real-rootedness and the log-concavities 
of coordinator polynomials of Weyl group lattices. 

Following Ardila et al.~\cite{ABHPS11}, 
we give an overview of the notions. 
Let $\mathcal{L}$ be a lattice, that is,
a discrete subgroup of a finite-dimensional Euclidean vector space
$E$. The dimension of the subspace spanned by $\mathcal{L}$ is called its
rank. A lattice is said to be generated as a monoid if there exists
a finite collection~$M$ of vectors such that every vector in the
lattice is a nonnegative integer linear combination
of the vectors in~$M$. 
Suppose that $\mathcal{L}$ is a lattice of rank~$d$, generated by~$M$. 
For any vector $v$ in~$\mathcal{L}$, define the length of~$v$ with respect to~$M$ 
to be the minimum sum of the
coefficients among all nonnegative integer linear combinations, denoted by~$\ell(v)$. In
other words,
\[
\ell(v)=\min\biggl\{\ \sum_{m\in M}c_m\ \, \bigg|\ \,v=\sum_{m\in M}
c_mm,\ c_m\geq 0\ \biggr\}.
\]
Let $S(k)$ be the number of vectors of length $k$ in $\mathcal{L}$.
Benson~\cite{Ben83} proved that the generating function
\begin{equation}\label{def-h}
\sum_{k\ge0}S(k)x^{k}={h(x)\over(1-x)^d}
\end{equation}
is rational, where $h(x)$ is a polynomial of degree at most $d$. 
Following Conway and Sloane~\cite{CS97},
we call $h(x)$ the {\em coordinator polynomial} with respect
to~$M$.

We concern ourselves with
the classical root lattices as $\mathcal{L}$. Let
$e_i$ denote the vector in $E$, having the $i$th entry one, and all
other entries zero, where the space $E$ is taken to
be~$\mathbb{R}^{n+1}$ for the root lattice $A_n$,
and to be~$\mathbb{R}^{n}$ for the root lattices $B_n$, $C_n$, and $D_n$.
The root lattices can be defined to be generated as monoids
respectively by
\begin{align*}
M_{A_n}&=\bigl\{\pm(e_i-e_j)\,\big|\,1\le i<j\le n+1\bigr\},\\[5pt]
M_{B_n}&=\bigl\{\pm e_i\pm e_j\,\big|\,1\le i<j\le n\bigr\}
\cup\bigl\{\pm e_i\,\big|\,1\le i\le n\bigr\},\\[5pt]
M_{C_n}&=\bigl\{\pm e_i\pm e_j\,\big|\,1\le i<j\le n\bigr\}
\cup\bigl\{\pm 2e_i\,\big|\,1\le i\le n\bigr\},\\[5pt]
M_{D_n}&=\bigl\{\pm e_i\pm e_j\,\big|\,1\le i<j\le n\bigr\}.
\end{align*}
We denote the coordinator polynomial of type $T$ by $h_T(x)$. 
Conway and Sloane~\cite{CS97} established the explicit
expression
\begin{equation}\label{a}
h_{A_n} (x) = \sum_{k=0}^{n}{n\choose k}^2 x^k,
\end{equation}
which were also called the
Narayana polynomials of type $B$ by Chen, Tang, Wang and
Yang~\cite{CTWY10}. In fact, these polynomials appeared as the rank
generating function of the lattice of noncrossing partitions of type
$B$ on the set $\{1,2,\ldots,n\}$.
With Colin Mallows's help, Conway and Sloane~\cite{CS97} conjectured that
\begin{equation}\label{d}
h_{D_n}(x)=
\frac{(1+\sqrt{x})^{2n}+(1-\sqrt{x})^{2n}}{2}-2n x(1+x)^{n-2}.
\end{equation}
Baake and Grimm~\cite{BG97} pointed out that the methods
outlined in~\cite{CS97} can be used to deduce that the coordinator
polynomials of type $C$ have the expression
\begin{equation}\label{c}
h_{C_n}(x)=\sum_{k=0}^{n}\binom{2n}{2k} x^k.
\end{equation}
They also conjectured that
\begin{equation}\label{b}
h_{B_n}(x)=\sum_{k=0}^{n} \binom{2n+1}{2k}x^k - 2n x(1+x)^{n-1}.
\end{equation}
Bacher, de la Harpe and Venkov~\cite{BHV97} rederived~\eqref{a} and
proved the formulas~\eqref{d}, \eqref{c} and~\eqref{b}. Recently,
Ardila et al.~\cite{ABHPS11} gave  alternative proofs
for~\eqref{a}, \eqref{d} and \eqref{c} by computing the
$f$-vectors of a unimodular triangulation of the corresponding root
polytope.

The real-rootedness of coordinator polynomials has
received much attention. As pointed out by Conway and
Sloane~\cite{CS97}, coordinator polynomials of type~$A$ can be
expressed as
\[
h_{A_n}(x)=(1-x)^{n}L_n\biggl(\frac{1+x}{1-x}\biggr),
\]
where $L_n(x)$ denotes the $n$th Legendre polynomial. 
Since Legendre polynomials are orthogonal, and thus real-rooted, 
we are led to the following result.
\begin{thm}\label{thm-A}
The coordinator polynomials of type $A$ are real-rooted.
\end{thm}
In fact, Theorem~\ref{thm-A} follows 
immediately from a classical result of Schur~\cite{Sch14},
see also Theorems 2.4.1 and 3.5.3 in Brenti~\cite{Bre89}. 
For $h_{C_n}(x)$, 
one may easily deduce the real-rootedness by
the Aissen-Schoenberg-Whitney theorem.

\begin{thm}\label{thm-C}
The coordinator polynomials of type $C$ are real-rooted.
\end{thm}

Liang and Yang~\cite{LY} reproved
both Theorems~\ref{thm-A} and~\ref{thm-C} by establishing 
recurrences of the coefficients. Moreover,
they verified the real-rootedness of coordinator polynomials 
of types $E_6$, $E_7$, $F_4$ and $G_2$.
In contrast, $h_{E_8}(x)$ is not real-rooted. 
They also conjectured that $h_{D_n}(x)$ is real-rooted. 

For coordinator polynomials of type $B_n$, we find 
$h_{B_{16}}(x)$ has $14$ real roots and $2$ non-real roots.
So $h_{B_n}(x)$ are not real-rooted in general.
In the next section, 
we develop a trigonometric substitution approach
which enables us to confirm Liang-Yang's conjecture. 
In Section 3,
we establish the log-concavities of all coordinator polynomials
of Weyl group lattices.

\section{The real-rootedness}

In this section, we show the real-rootedness of $h_{D_n}(x)$.

\begin{thm}\label{thm-D}
The coordinator polynomials of type $D$ are real-rooted.
\end{thm}

We shall adopt a technique of trigonometric transformation.
To be precise, we transform the polynomial $h_{D_n}(x)$ into a
trigonometric function, say, $g_n(\theta)$,
and then consider the roots of $g_n(\theta)$. 
It turns out that the signs
of $g_n(\theta)$ at a sequence of $n+1$ fixed values of $\theta$ 
are interlacing. Hence $g_n(\theta)$ 
has $n$ distinct zeros in a certain domain, and so does $h_{D_n}(x)$. 

\begin{proof}
We are going to show that the polynomial $h_{D_n}(x)$
has $n$ distinct negative roots for $n\ge2$.
For this purpose, we let $y>0$ and substitute $x=-y^2$ in the expression~\eqref{d} of $h_{D_n}(x)$.
Note that $\sqrt{-y^2}$ is two-valued, denoting $\pm yi$. 
However, taking $\sqrt{-y^2}=yi$ and taking $\sqrt{-y^2}=-yi$ 
yields the same expression of $h_{D_n}(-y^2)$, that is,
\begin{equation}\label{eq-f}
h_{D_n}(-y^2)={(1+yi)^{2n}+(1-yi)^{2n}\over2}+2ny^2(1-y^2)^{n-2}.
\end{equation}

Without loss of generality, we can suppose that 
\[
y=\tan{\phi\over2},
\]
where $\phi\in(0,\pi)$. Then $1+yi=\sqrt{1+y^2}\,e^{i\phi/2}$,
and thus
\[
(1+yi)^{2n}+(1-yi)^{2n}=(1+y^2)^ne^{in\phi}+(1+y^2)^ne^{-in\phi}=2(1+y^2)^n\cos{n\phi}.
\]
It follows that
\[
h_{D_n}(-y^2)=(1+y^2)^n\cos{n\phi}+2ny^2(1-y^2)^{n-2}
=\bigl(1+y^2\bigr)^ng_n(\phi),
\]
where
\begin{equation}\label{g}
g_n(\phi)=\cos{n\phi}+{n\over2}\sin^2\phi\cos^{n-2}\phi.
\end{equation}

Now it suffices to prove that the function $g_n(\phi)$ 
has $n$ distinct roots $\phi$ in the interval $(0,\pi)$. 
Let
\[
h_n(\phi)={n\over2}\sin^2\phi\cos^{n-2}\phi.
\]
We claim that 
\begin{equation}\label{h}
|h_n(\phi)|<1.
\end{equation}
In fact, by the arithmetic-geometric mean inequality, 
\begin{align}
h_n^2(\phi) & =
\Bigl({n\over2}\sin^2\phi\Bigr)\Bigl({n\over2}\sin^2\phi\Bigr)
\bigl(\cos^2\phi\bigr)^{n-2}\notag \\
& \le
\Bigl({{n\over2}\sin^2\phi+{n\over2}\sin^2\phi+(n-2)\cos^2\phi
\Bigr)}^n\Big/n^n\label{eq-11}\\[5pt]
&=\Bigl(1-{2\cos^2\phi\over n}\Bigr)^n\le1.\label{eq-12}
\end{align}
Note that the equality in~\eqref{eq-11} holds if and only if
\begin{equation}\label{cond1}
{n\over2}\sin^2\phi=\cos^2\phi,
\end{equation}
while the equality in~\eqref{eq-12}
holds if and only if 
\begin{equation}\label{cond2}
\cos\phi=0. 
\end{equation}
However, the conditions~\eqref{cond1} and~\eqref{cond2}
contradict each other. So the equality in~\eqref{eq-12}
does not hold. This confirms the claim~\eqref{h}.

Let
$j$ be an integer. From~\eqref{g}, we see that
\[
g_n\Bigl({j\pi\over n}\Bigr)=(-1)^j+h_n\Bigl({j\pi\over n}\Bigr).
\]
Since $|h(\phi)|<1$ for any $\phi$, we have
\[
(-1)^jg_n\Bigl({j\pi\over n}\Bigr)
=1+(-1)^jh_n\Bigl({j\pi\over n}\Bigr)>0.
\]
By its continuity, we obtain that $g_n(\phi)$ has roots
$\phi_0,\phi_1,\ldots,\phi_{n-1}$ such that
\[
0<\phi_0<{\pi\over n}<\phi_1<{2\pi\over n}<\phi_2<{3\pi\over n}
<\cdots<{(n-1)\pi\over n}<\phi_{n-1}<\pi.
\]
In conclusion, the polynomial $h_{D_n}(x)$ has $n$ distinct negative roots
\[
x_j=-\tan^2{\phi_j\over2},\qquad j=0,1,\ldots,n-1.
\]
This completes the proof.
\end{proof}

\section{The log-concavity}

In this section,
we consider the log-concavities
of coordinator polynomials of Weyl group lattices.
For basic notions on the Weyl group, 
see Humphreys~\cite{Hum72B}.
By the definition~(\ref{def-h}), it is easy to see that
the coordinator polynomial of any Weyl group lattice
is the product of the coordinator polynomials of 
the Weyl group lattices determined by the irreducible components.
This has been noticed by, for instance, Conway and Sloane~\cite[Page 2373]{CS97}.
By the Cauchy-Binet theorem,
the product of log-concave polynomials with nonnegative coefficients and no
internal zero coefficients are log-concave;
see Stanley~\cite[Proposition 2]{Sta89}.
Therefore, we are led to consider the coordinator polynomials of 
the Weyl group lattices which are determined by irreducible root systems.
By the Cartan-Killing classification, 
irreducible root systems can be classified into types $A_n$, $B_n$, $C_n$, $D_n$, 
$E_6$, $E_7$, $E_8$, $F_4$ and $G_2$.
For historical notes, see Bourbaki~\cite{Bou02}.
To conclude, we have the following result.

\begin{thm}
All coordinator polynomials of Weyl group lattices are log-concave. 
\end{thm}

\begin{proof}
It is straightforward to verify the log-concavity of the coordinator polynomials 
of types $E_6$, $E_7$, $E_8$, $F_4$ and $G_2$.
By Theorems~\ref{thm-A}, \ref{thm-C} and~\ref{thm-D}, 
it suffices to show the log-concavity of $h_{B_n}(x)$.
Let $b_k$ be the coefficient of $x^k$ in $h_{B_n}(x)$,
and let $b_k'=b_k\big/{n\choose k}$.
By~(\ref{b}), we have
\[
b_k'={(2n+1)!!\over (2k-1)!!(2n-2k+1)!!}-2k.
\]
It is easy to verify that the sequence $\bigl\{b_k'\bigr\}_{k=0}^n$ is log-concave,
which implies the log-concavity of $\bigl\{b_k\bigr\}_{k=0}^n$.
This completes the proof.
\end{proof}

\noindent{\bf Acknowledgments.}
This work was supported by
the National Natural Science Foundation of China
(Grant No.~$11101010$).
We are grateful to Arthur Yang 
for his kindly telling us the real-rootedness conjecture,
and to Chunwei Song for his encouragements.
We also thank the anonymous referee for detailed comments
that improved the organization of this material.

\end{document}